\documentclass[12pt]{amsart}
\usepackage{amssymb}
\usepackage{amsmath}
\usepackage{bbm}
\usepackage[dvips]{graphicx}
\usepackage{hyperref}
\usepackage[dvipsnames]{xcolor}
\usepackage[capitalize,nameinlink,noabbrev,nosort]{cleveref}

\usepackage{tikz}
\usetikzlibrary{arrows}
\usetikzlibrary{decorations.markings}
\tikzset{invmidarrow/.style={postaction=decorate,decoration={markings,mark={between positions 0 and 1 step #1
with {\arrow{latex reversed}}}}}}
\usetikzlibrary{shapes.geometric}

\numberwithin{equation}{section}

\newcommand{\ds}{\displaystyle}
\newcommand{\ZZ}{\mathbb{Z}}
\newcommand{\NN}{\mathbb{N}}
\newcommand{\RR}{\mathbb{R}}
\newcommand{\tA}{{\widetilde{A}}}
\newcommand{\tB}{{\widetilde{B}}}
\newcommand{\cC}{\mathcal{C}}
\newcommand{\tS}{{\widetilde{S}}}
\newcommand{\tR}{{\tilde{R}}}
\newcommand{\aver}[1]{\langle #1 \rangle}
\newcommand{\dens}[1]{\langle \eta_{#1} \rangle}
\newcommand{\vac}{\circ}
\newcommand{\occ}{\bullet}
\DeclareMathOperator{\ASEP}{f}
\DeclareMathOperator{\wt}{wt}
\newcommand{\hOmega}{{\widehat\Omega}}
\newcommand{\heta}{{\hat\eta}}
\newcommand{\var}{x}
\newcommand\tcb[1]{\textcolor{blue}{#1}}

\newtheorem{thm}{Theorem}
\newtheorem{prop}[thm]{Proposition}
\newtheorem{cor}[thm]{Corollary}
\newtheorem{lem}[thm]{Lemma}

\newtheorem{rem}[thm]{Remark}

\let\origmaketitle\maketitle
\def\maketitle{
  \begingroup
  \def\uppercasenonmath##1{} % this disables uppercasing title
  \origmaketitle
  \endgroup
}

\title
[The multispecies PushTASEP: Dynamics and symmetry]
{The inhomogeneous multispecies PushTASEP: Dynamics and symmetry
}
\author{Arvind Ayyer}
\address{Arvind Ayyer, Department of Mathematics, 
Indian Institute of Science, Bangalore  560012, India.}
\email{arvind@iisc.ac.in}

\author[J.~B.~Martin]{James B.\ Martin}
\address{James B.~Martin, Department of Statistics, University of Oxford, UK}
\email{martin@stats.ox.ac.uk}

\date{\today}

\begin{document}

\begin{abstract}
We introduce and study a natural multispecies variant of the inhomogeneous PushTASEP with site-dependent rates on the finite ring. We show that the stationary distribution of this process is proportional to the ASEP polynomials at $q = 1$ and $t = 0$. 
This is done by constructing a multiline process which 
projects to the multispecies PushTASEP, 
and identifying its stationary distribution using
time-reversal arguments. We also study 
symmetry properties of the process under interchange of the 
rates associated to the sites. These results hold not just 
for events depending on the configuration at a single time 
in stationarity, but also for systems out of equilibrium and 
for events depending on the path of the process over time.
Lastly, we give explicit formulas for nearest-neighbour two-point correlations in terms of Schur functions.
\end{abstract}


\maketitle

\section{Introduction}

The multispecies asymmetric simple exclusion process (ASEP)
is an interacting particle system originating in statistical mechanics
and probability, with significant connections to different areas of mathematics such as combinatorics and representation theory. 
It has been the focus of considerable study in the last few decades. 
The stationary distribution of the multispecies TASEP (i.e.\ 
ASEP with asymmetry parameter $t=0$) on a ring of size $n$ was 
described by Ferrari and Martin~\cite{ferrari-martin-2007} in terms of multiline diagrams. The connection between the multispecies ASEP with content $\lambda$ and the Macdonald polynomial $P_\lambda(x_1, \dots, x_n; q, t)$ was realised by 
Cantini--de Gier--Wheeler~\cite{cantini-degier-wheeler-2015} and 
by Chen--de Gier--Wheeler~\cite{chen-degier-wheeler-2020}. In these works, they define ASEP polynomials (which are called relative Macdonald polynomials in~\cite{guo-ram-2021}, and 
which are special cases of the
permuted basement Macdonald polynomials studied in \cite{ferreira-2011,alexandersson-2019}).
These are polynomials in variables $x_1, \dots, x_n$ whose coefficients are rational functions in $q, t$, which, for $q=1$ and all $x_i$ equal, 
are proportional to the stationary distribution of the multispecies ASEP. Later on, Corteel--Mandelshtam--Williams~\cite{corteel-mandelshtam-williams-2022} showed that the ASEP polynomials can be computed using the multiline diagrams stated above by assigning a weight depending on $q, t$ and the variables $x_1, \dots, x_n$.

The ASEP itself is not thought to have nice algebraic 
properties when the jump rates are allowed to differ between sites -- 
but it is natural to ask if there is some multispecies particle system 
with site-wise inhomogeneity 
whose stationary distribution is proportional to the ASEP polynomial 
at $q=1$ for general values of the parameters $x_1, \dots, x_n$. 

In this and a companion paper \cite{ayyer-martin-williams-2024}, we show that this is indeed the case, for a multispecies inhomogeneous version of the \textit{PushTASEP}, with deformation parameter $t\geq 0$. The case of general $t$ is covered
in \cite{ayyer-martin-williams-2024}. In this work, we focus on 
various interesting aspects of the particular case $t=0$. 

The arguments in \cite{ayyer-martin-williams-2024} involve
developing relationships between the ASEP polynomials and the 
non-symmetric Macdonald polynomials. 
Here for the case $t=0$, we explain an alternative proof of a 
different nature, constructing a Markov chain on multiline
diagrams whose 
bottom row projects to the PushTASEP process, and whose stationary distribution is proportional to the weight defined 
in \cite{corteel-mandelshtam-williams-2022}. This 
extends the approach of 
\cite{ferrari-martin-2006} 
to the inhomogeneous setting. 

A particular focus is on symmetry properties 
under interchange of the rates between different sites. 
We show that probabilities of events depending
on a given interval of sites remain unchanged
when the rates at other sites are permuted. 
Such results are proved for the general case $t\geq 0$
in \cite{ayyer-martin-williams-2024},
for events depending on the configuration at a single time
in the stationary distribution.
In the special case $t=0$ we can prove significantly
stronger properties, applying to systems out of equilibrium,
and to events which depend not just on a single time but on 
the path of an evolving process. These results are obtained 
by using coupling methods to prove
interchangeability results for PushTASEP stations 
analogous to those previously proved for exponential queueing servers
\cite{Weber1979,TsoucasWalrand1987,MartinPrabhakar2010}.

Finally we study various observables for the system 
in stationary, such as the density of particles of 
a given species at a given site, and the current between sites
for particles of a given species. We give explicit formulas for nearest-neighbour two-point correlations in terms of Schur functions.

The PushTASEP, although somewhat less widely known than its cousin
the TASEP, has nonetheless been widely studied in recent years. 
Spitzer~\cite{spitzer-1970} introduced it as an example of a \emph{long-range exclusion process}, but more recently the name PushTASEP  has become established.
It has often been studied on the infinite lattice $\mathbb{Z}$ -- see for example \cite{petrov-2020} for extensive references -- 
but there have been a few studies on finite graphs as well~\cite{ayyer_schilling_steinberg_thiery.sandpile.2013,ayyer-2016}.
The multispecies case was already considered in \cite{ferrari-martin-2006}, under the guise of the discrete-space Hammersley--Aldous--Diaconis
process, to which it is equivalent by particle-hole duality.
A related multispecies process in discrete time (dubbed the ``frog model'') was recently used by Bukh--Cox~\cite{bukh-cox-2022}
to study problems involving the longest common subsequence between a periodic word and a word with i.i.d.\ uniform entries. 
For an inhomogeneous version on the line, limit shape and fluctuation results are obtained by Petrov \cite{petrov-2020}. The multispecies inhomogeneous model with $t>0$ was studied on an open interval by Borodin and Wheeler \cite[Section 12.5]{borodin-wheeler-2022}.

As this project was being completed, related work by 
Aggarwal, Nicoletti and Petrov appeared \cite{aggarwal-nicoletti-petrov-2023}, which 
studies an inhomogeneous $t$PushTASEP on the ring 
(in fact, a more general model in which the sites can have
capacities greater than $1$). They obtain a description
of the stationary distribution of such systems in terms
of ``queue vertex models'' on the cylinder -- these objects
are closely related to multiline diagrams (and to matrix
product formulae). Their methods are very different to ours, 
making extensive use of the Yang-Baxter equation. 

\subsection{Main results}

We now describe the model we study, an
inhomogeneous PushTASEP on a 
ring of size $n$, with $s$ species of particle (called $1,2,\dots, s$).

We start with the single species case, i.e.\ $s=1$. 
Each site contains either a single particle of type $1$, 
or a vacancy. The number of particles will be conserved by the 
dynamics -- write $m_1$ for the number of particles, and $m_0:=n-m_1$
for the number of vacancies. Then a configuration
$\eta=(\eta_j, 1\leq j \leq n)$ is a tuple of length $n$ containing
$m_1$ $1$s and $m_0$ $0$s, where $\eta_i=1$ if the configuration
has a particle at $i$. 

We have positive real parameters $\var_1, \dots, \var_n$,
which we think of as attached to the sites $1,2,\dots, n$. 
The system is a Markov chain whose rates can be described as follows.
At each site $j$, a bell rings at rate $1/\var_j$.
When this bell rings, if $j$ is vacant, nothing changes. If
$j$ is occupied by a particle, that particle moves to the first vacant site clockwise from $j$, leaving a vacancy at $j$ itself.

Now we move to the general case of $s$ species. We denote
$\eta_j=r$ if the configuration $\eta$ has a particle of type $r$
at site $j$, 
where $1\leq r\leq s$, and $\eta_j=0$ if site $j$ is vacant. 
The number of particles of each type will be conserved by the 
dynamics: write $m_r$ for the number of particles of species $r$
for $1\leq r\leq s$ -- now the number of vacancies is $m_0
= n-\sum_{r=1}^s m_r$. 

Equivalently we can describe the particle content 
as a partition (i.e.\ a weakly decreasing tuple of non-negative integers)
of length $n$,
given by 
\[
\lambda = (\underbrace{s, \dots, s}_{m_s}, \dots,
\underbrace{1, \dots, 1}_{m_1},
\underbrace{0, \dots, 0}_{m_0}),
\]
It will also be convenient to write $\lambda$ in 
frequency notation as $\lambda = \langle
0^{m_0}, 1^{m_1}, \dots, s^{m_s} \rangle$.
Unless otherwise stated, we will fix throughout 
some partition $\lambda$ giving the content of the system,
with $s$ species of particles and length $\lambda$. 

A configuration is then a permutation of $\lambda$,
and may be thought of as a composition (i.e.\ a tuple of non-negative integers). We write $\Omega_\lambda$ for the set
of all configurations, which is the state-space of 
the Markov chain. For example,
\[
\Omega_{(2,1,0)} = \{210, 201, 120, 102, 021, 012\}.
\]

The 
dynamics of the multispecies chain are as follows;
the idea is that 
higher-numbered species are ``stronger'' than 
lower-numbered species, and can displace them. 
As above, a bell rings at each site $j$ at rate $1/\var_j$. 
When such a bell rings, if $\eta_j=0$, i.e.\ the site $j$ is vacant, nothing changes. 
Otherwise, suppose $j$ contains a particle of species $i_1$. 
Then this particle moves to the first location $j_2$ clockwise 
from $j$ with $i_2:=\eta_{j_2}<i_1$. If this site $j_2$ 
was previously vacant (i.e.\ $i_2=0$), the jump is complete; 
otherwise the particle of type $i_2$ itself moves clockwise 
until it finds a site $j_3$ with $\eta(j_3)<i_2$. This game of `musical chairs' continues, with stronger particles displacing weaker ones in turn, until a vacancy is found. All the particles involved move 
instantaneously to their chosen positions, leaving a 
vacancy at site $j$. 

An important observation is that the multispecies dynamics with
$s$ species of particle can be seen as a coupling
(the so-called ``basic coupling'') of $s$ single-type
PushTASEP processes. These can be obtained by regarding
all species of type $r$ or higher as ``particles'' 
and all species of type $r-1$ or lower as ``vacancies'', 
for any $1\leq r\leq s$. See \cref{prop:colouring} at the end of \cref{sec:single}
for more details. 

To illustrate the dynamics, consider the configuration $\eta = (2, 0, 1, 4, \allowbreak 2, 0, 3, 1)$ on $n = 8$ sites with $s = 4$ species. Then
for example 
\begin{equation}
\label{multilrep-trans-eg}
\eta \to 
\begin{cases}
(2,0,1,0,4,2,3,1) & \text{with rate $1/\var_4$}, \\
(2,1,1,4,2,0,0,3) & \text{with rate $1/\var_7$}.
\end{cases}
\end{equation}
If a bell rings either at site $2$ or at site $6$ rings,
nothing changes. 

The transition graph of the multispecies PushTASEP on $\Omega_{(2,1,0)}$ is given in \cref{fig:eg123}.

\begin{center}
\begin{figure}
\includegraphics[scale=0.5]{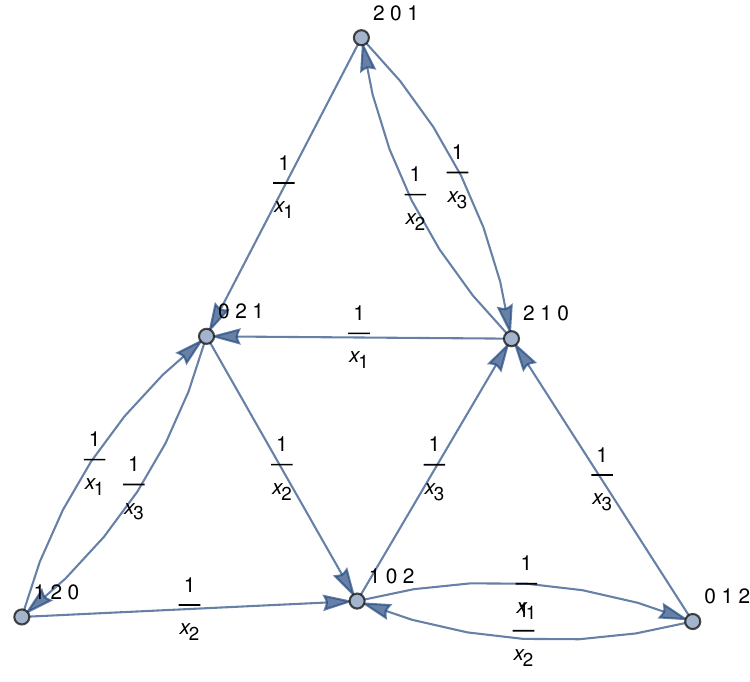}
\caption{The transition graph of the multispecies PushTASEP with 
$\lambda = (2,1,0)$.}
\label{fig:eg123}
\end{figure}
\end{center}

Note that when a bell rings at a given site in a given configuration,
the transition it triggers is deterministic. (This will no longer
be true in the case $t>0$.)

We start with an explicit formula for the stationary distribution. To state it, we first recall that the
{\em elementary symmetric polynomial} $e_m$ for $m \leq n$ is given by
\begin{equation}
\label{def-elemsymfn}
e_m(x_1, \dots, x_n) = \sum_{1 \leq i_1 < \cdots < i_m \leq n}
x_{i_1} \dots x_{i_m}.
\end{equation}
For a partition $\mu$, we define $e_\mu = \prod_i e_{\mu_i}$.
We also need the notion of the \emph{ASEP polynomials}, first defined in Cantini--de Gier--Wheeler~\cite{cantini-degier-wheeler-2015} 
and later named as such by Chen--de Gier--Wheeler~\cite{chen-degier-wheeler-2020}, denoted $\ASEP_\eta(x_1, \allowbreak  \dots, x_n; q, t)$. These are important nonsymmetric polynomials related to the well-known symmetric Macdonald polynomials $P_\lambda(x_1, \dots, x_n; q, t)$.
The ASEP polynomials, indexed by compositions $\eta$, are 
polynomials in the variables $x_1, \dots, x_n$,
whose coefficients are rational functions in $q, t$.
We express the stationary distribution for a system with content 
$\lambda$ using the family of ASEP polynomials $\ASEP_\eta$ where $\eta$ is a permutation of $\lambda$.

\begin{thm}
\label{thm:multilrep-statprob}
Let $\lambda=\langle
0^{m_0}, 1^{m_1}, \dots, s^{m_s} \rangle$ be a partition, 
and set $M_i = m_{i} + \cdots + m_s$ for $1 \leq i \leq s$.
The stationary distribution $\pi$ of
the multispecies PushTASEP with content $\lambda$ is given by 
\[
\pi(\eta) = \frac{\ASEP_\eta(\var_1, \dots, \var_n; q = 1, t = 0)}{Z_\lambda},
\]
for $\eta\in\Omega_\lambda$,
where $\ASEP_\eta$ is the ASEP polynomial and
\[
Z_\lambda(\var_1,\dots,\var_n) = P_\lambda(\var_1, \dots, \var_n; 1, 0) = 
\prod_{i=1}^{s-1} e_{M_i}(\var_1,\dots,\var_n)
\]
is the {\em partition function}.
\end{thm}

We discuss the ASEP polynomials further in \cref{sec:multiline}, although we don't give a full definition. 
For a definition in full generality, see for example 
\cite{chen-degier-wheeler-2020, corteel-mandelshtam-williams-2022, 
ayyer-martin-williams-2024} for various approaches. 
If the reader prefers, \cref{thm:multilrep-statprob}
can instead be understood without any reference to ASEP 
polynomials, as a statement that the stationary probabilities
are proportional to sums of weights of multiline diagrams,
as defined in \cref{sec:multiline}.
See \cref{thm:stat prob sum} for a precise 
statement of the result in this form. 

Note that in the homogenous case $\var_1=\dots=\var_n$,
the stationary distribution of the system is the same
as for the multispecies TASEP, (as was already observed, 
more specifically for the system on the whole line, in 
\cite{ferrari-martin-2006}).

\cref{thm:stat dist multiline} is a special 
case of the result proved in the more general case 
of the $t$-PushTASEP for $t\geq 0$ in \cite{ayyer-martin-williams-2024}, via
relationships between ASEP polynomials and 
non-symmetric Macdonald polynomials; however the more probabilistic approach explained here via 
the multiline process and time-reversal also seems interesting in 
its own right. 

Our second main result is about the invariance of the process under permutations of the parameters $(\var_1, \dots, \var_n)$. 

\begin{thm}
\label{thm:multipath}
Consider the multispecies PushTASEP on the ring with content $\lambda$,
started either in the stationary distribution,
or in any starting configuration with 
$\eta \in \Omega_\lambda$ in which 
$\eta_{k+1} \geq \eta_{k+2}\geq\dots \eta_n$. 

The distribution of the path of the process
observed on sites $1,2,\dots, k$ is 
invariant under permutations of the parameters
$\var_{k+1}, \dots, \var_n$. 
\end{thm}

An analogous property was recently proved for the multispecies TAZ\-RP~\cite{ayyer-mandelshtam-martin-2022} using results about interchangeability of exponential-server queues in series. To 
prove \cref{thm:multipath}, we need to develop 
similar interchangeability results in a new context where
the exponential-server queues are replaced by PushTASEP stations.

A related symmetry property for events depending just on the 
configuration at a single time in stationarity is obtained
in the general case $t\geq 0$ in 
\cite{ayyer-martin-williams-2024}. However the result
of 
\cref{thm:multipath} is much stronger, 
both since it applies to systems out of stationarity,
and since it applies to events which depend on
the path of the process evolving over time.
The question of whether this stronger symmetry
property can also be extended to $t>0$ seems very interesting.

Finally, 
we look at various important observables 
for the PushTASEP stationary distribution. We will give formulas for the density and current of particles in \cref{sec:obs},
which can be derived from the corresponding results
for a single species system.
Unlike in the homogeneous case, even
the density of a given species at a given site is
not obvious.
A more intricate result is a formula for the nearest-neighbour
two-point correlation, which is our third main result. This generalizes results of Ayyer--Linusson for the multispecies TASEP~\cite{ayyer-linusson-2016} and Amir--Angel--Valko for the TASEP speed process~\cite{amir-angel-valko-2011}.

For the purposes of this discussion, we will restrict ourselves to $\lambda = (n-1,\dots,1,0)$ with $n-1$ species of particles
(but all other cases can be derived from this case
by projection). 

To state the result, we recall that 
the {\em Schur polynomial} $s_\lambda(x_1, \dots, x_n)$ are an important family of symmetric polynomials playing a significant role in representation theory and algebraic geometry~\cite{stanley-ec2}. 
The Schur polynomial $s_\lambda$ can be defined as a ratio of determinants,
\begin{equation}
\label{def-schur}
s_\lambda(x_1, \dots, x_n) = \frac{\det (x_i^{\lambda_j + n - j} )}{\det (x_i^{n - j} )}.
\end{equation}
which is symmetric in $x_1,\dots, x_n$. 
For $1 \leq i < j \leq n$, define the polynomials
\begin{multline*}
f_{j,i}(\var_1,\dots,\var_n) = \\
\det \begin{pmatrix}
1 & 
s_{\langle 2^{n-j-2} \rangle}(\var_3,\dots,\var_n) & 
s_{\langle 2^{n-i-2} \rangle}(\var_3,\dots,\var_n) \\
- \var_1 - \var_2 & 
s_{\langle 1^1, 2^{n-j-2} \rangle}(\var_3,\dots,\var_n) & 
s_{\langle 1^1, 2^{n-i-2} \rangle}(\var_3,\dots,\var_n) \\
\var_1 \var_2 & 
s_{\langle 2^{n-j-1} \rangle}(\var_3,\dots,\var_n) & 
s_{\langle 2^{n-i-1} \rangle}(\var_3,\dots,\var_n)
\end{pmatrix},
\end{multline*}
and
\begin{multline*}
g_{j,i}(\var_1,\dots,\var_n) = \sum_{a,b=0}^2 
\var_1^{3-a} \var_2^{3-b}
s_{(2^{n-i-2}, a)}(\var_3,\dots,\var_n) \\
\times
s_{(2^{n-j-2}, b)}(\var_3,\dots,\var_n),
\end{multline*}
where $s_{\langle 1^{m_1}, 2^{m_2} \rangle} = 0$ if $m_2$ is negative.

Let $\eta^{(i)}_k$ denote the occupation variable for the particle of species $i$ at site $k$, i.e., $\eta^{(i)}_k = 1$ (resp. $\eta^{(i)}_k=0$) when the $k$th site is occupied (resp. not occupied) by 
a particle of species $i$ 
(where ``particle of species $0$'' means a vacancy).

\begin{thm}
\label{thm:2pt}
Let $0 \leq i,j \leq n-1$. Then the joint probability of seeing the particle of species $j$ at site $1$ and that of species $i$ at 
site $2$ is given by
\[
\aver{\eta^{(j)}_1 \eta^{(i)}_2} = 
\begin{cases}
\ds \var_1 \var_2^2 \frac{f_{j,i}(\var_1,\dots,\var_n)}
{e_{(n-j, n-j-1, n-i, n-i-1)}(\var_1,\dots,\var_n)} &
j < i, \\
\noalign{\vskip9pt}
0 & j = i, \\
\noalign{\vskip9pt}
\ds \frac{g_{i+1,i}(\var_1,\dots,\var_n)} 
{e_{(n-i, n-i-1, n-i-1, n-i-2)}(\var_1,\dots,\var_n) } & \\
\noalign{\vskip9pt}
\hspace*{1cm} \ds + \var_1 \var_2 \frac{s_{\langle 2^{n-j-1} \rangle}(\var_3,\dots,\var_n)}{e_{(n-j,n-j)}(\var_1,\dots,\var_n)} & j = i+1, \\
\noalign{\vskip9pt}
\ds \frac{g_{i,j}(\var_1,\dots,\var_n)}
{e_{(n-i, n-i-1, n-j, n-j-1)}(\var_1,\dots,\var_n) } &
j > i+1.
\end{cases}
\]
\end{thm}

The proof of \cref{thm:2pt} involves projecting
from the multispecies system to a suitable
$3$-species system, for which the probabilities
of relevant events can be analysed using $2$-line diagrams. 

Note that in all cases the two-point correlation function obtained in 
\cref{thm:2pt} is symmetric in $\var_3,\dots, \var_n$,
as must be the case in the light of \cref{thm:multipath}. 
It is interesting to observe the appearance of the 
Schur polynomials in these expressions. 
In the homogeneous case, under suitable rescaling in the limit $n\to\infty$, one obtains interesting asymptotics for the 
joint distribution of the species observed at the two sites as shown in \cite{amir-angel-valko-2011};
setting $(x,y)=(j/n, i/n)$, one obtains a limit with constant density 
in the region $x>y$, a singular term involving non-vanishing probability on the boundary
$x=y$, and repulsion from the boundary in the region $x<y$. 
It would be interesting to explore asymptotics in 
regimes involving inhomogeneous parameters $\var_1,\dots, \var_n$.

Note that the result of  
\cite{ayyer-linusson-2016} has recently been extended in a 
different direction by Pahuja~\cite{pahuja-2023}, who obtains
nearest-neighbour two-point correlations for the multispecies ASEP 
(i.e.\ in our notation the case $t>0$ with all $\var_j$ equal),
based on the multiline diagram construction for the ASEP 
in \cite{martin-2020}.

The plan of the rest of the paper is as follows. 
Throughout, we consider the case of a multispecies
PushTASEP whose contents are fixed and given by some 
partition $\lambda$.
In \cref{sec:single}, we will derive the stationary distribution for the single species inhomogeneous PushTASEP as well as formulas for the density and current; we also explain the interpretation of the
multispecies system as a basic coupling of single species systems.
In \cref{sec:single symm}, we give a proof of the interchangeability of rates for the single species PushTASEP and use that to prove \cref{thm:multipath}. In \cref{sec:multiline}, we
define multiline diagrams and 
construct a multiline process 
which will lump to the multispecies PushTASEP, in order to 
obtain the stationary distribution in \cref{thm:multilrep-statprob}. Finally, in \cref{sec:obs}, we will prove formulas for the current and the density as well as the formula for the two-point correlations in \cref{thm:2pt}.

\section{Single species PushTASEP}
\label{sec:single}

In this section, we focus on the PushTASEP with a single species of particle, i.e. $s = 1$. As before, we take $n$ sites. Therefore, $\lambda = \langle 0^{m_0}, 1^{m_1} \rangle$ with $m_1 < n$ particles (and so $m_0 = n - m_1$ vacancies) and periodic boundary conditions.

Recall that  $\Omega_{\lambda}$ is the configuration space. 
For example, consider the case $m_0 = m_1 = 2$. With the lexicographic ordering of the configurations, i.e.,
\[
\Omega_{(1, 1, 0, 0)} = \{0011, 0101, 0110, 1001, 1010, 1100\},
\]
the transition graph of the chain is given in \cref{fig:eg0011}.

\begin{center}
\begin{figure}
\includegraphics[scale=0.5]{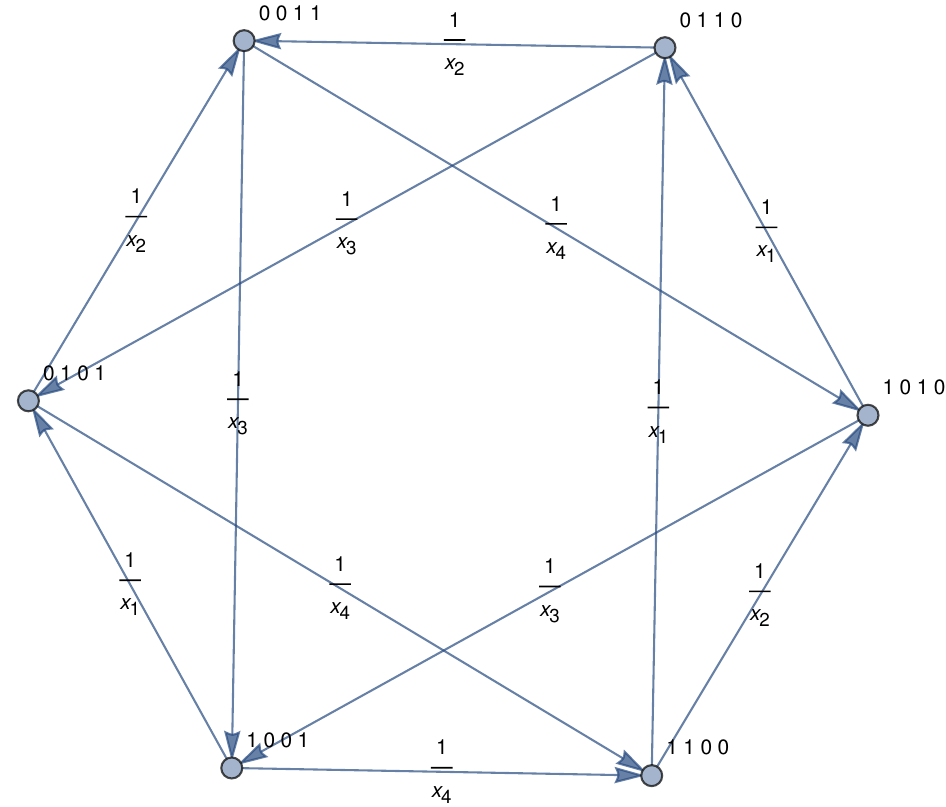}
\caption{The transition graph of the single species PushTASEP with $m_0 = m_1 = 2$.}
\label{fig:eg0011}
\end{figure}
\end{center}

We are interested in the stationary distribution of this process. 
It is easy to see that the single species PushTASEP is irreducible. Therefore, the stationary distribution is unique.
The the following proposition is straightforward to prove using the master equation. 

\begin{prop}
\label{prop:ss-singlePushTASEP}
The stationary probability $\pi$ of $\eta \in \Omega_{\langle 0^{m_0}, 1^{m_1} \rangle}$ for the PushTASEP is 
\[
\pi(\eta) = \frac{1}{e_{m_1}(\var_1,\dots,\var_{n})} \prod_{\substack{i=1 \\ \eta_i = 1} }^{n} \var_i.
\]
\end{prop}

Recall that the \emph{density} at site $i$ is the probability of finding a particle at site $i$ in the stationary distribution. 
We will use angular brackets $\langle \cdot \rangle$ to denote expectations in the stationary distribution. 
As is usual, we let $\eta_j$ be the indicator variable for a particle at site $j$, so that $\dens{j}$ is the density at site $j$.
Because of the factorized nature of the stationary distribution in \cref{prop:ss-singlePushTASEP}, we immediately obtain the following.

\begin{cor}
\label{cor:density-singlePushTASEP}
The density at site $j$ in the stationary distribution is given by
\[
\dens{j} = \frac{\var_j e_{m_1-1}(\var_1,\dots,\widehat{\var_j},
\dots, \var_{n}) }{e_{m_1}(\var_1,\dots,\var_{n})},
\]
where the hat on an argument denotes its absence.
\end{cor}

The \emph{current} of a particle across a given edge (say $(n,1)$) is the number of particles per unit time that cross that edge 
in stationarity. Because of particle conservation, the current is the same for all edges. We will denote the current by $J$.
In terms of the stationary distribution for the PushTASEP, this is given by
\begin{equation}
\label{def-curr}
J = \sum_{j = m_0 + 1}^n \frac 1{\var_j} \aver{ \eta_j \cdots \eta_n }.
\end{equation}

\begin{prop}
\label{prop:curr-singlePushTASEP}
The current in the stationary distribution is given by
\[
J = \frac{e_{m_1-1}(\var_1,\dots, \var_{n}) }{e_{m_1}(\var_1,\dots,\var_{n})}.
\]
\end{prop}

\begin{proof}
Using arguments similar to the proof of \cref{cor:density-singlePushTASEP}, it is easy to see that
\[
\aver{ \eta_j \dots \eta_n } = \var_j \cdots \var_n
\frac{e_{j-m_0-1}(\var_1,\dots, \var_{j-1}) }{e_{m_1}(\var_1,\dots,\var_{n})}.
\]
Plugging this into \eqref{def-curr}, we obtain, after setting $k = j-m_0-1$,
\begin{equation}
\label{curr1}
J = \sum_{k = 0}^{m_1-1} \var_{m_0 + 2 + k} \cdots \var_n
\frac{e_{k}(\var_1,\dots, \var_{m_0+k}) }{e_{m_1}(\var_1,\dots,\var_{n})}.
\end{equation}
We now note an elementary recursive formula for the elementary symmetric function $e_m(x_1, \dots, x_k)$ in \eqref{def-elemsymfn}. First, split this into terms according to whether $x_k$ appears or not to obtain
\[
e_m(x_1, \dots, x_k)  = e_m(x_1, \dots, x_{k-1}) + x_k e_{m-1}(x_1, \dots, x_{k-1}).
\]
Now, successively apply this recurrence to each variable starting from $x_{k-1}$ and progressing up to $x_1$, to get
\[
e_m(x_1, \dots, x_k)  = \sum_{i=k+1-m}^{k+1} x_i \cdots x_k \, e_{m-k-1+i}(x_1, \dots, x_{i-2}).
\]
One can check that the right hand side of this equation, after appropriate change of variables can be applied to \eqref{curr1}. This gives the desired result.
\end{proof}

We have the following projection property for 
our multispecies dynamics. If we view all 
particles of species $r,\dots, s$  as ``particles'', 
and all particles of lower species $1,\dots, r-1$ as vacancies, 
then we obtain a single species process with $M_i=m_i+\dots+m_s$
particles and $n-M_i$ vacancies. 

Considering such projects for all $r=1,2,\dots, s$, 
we can see the multispecies process as a coupling of $s$ 
single species processes. This is the \textit{basic coupling}
\cite[Chapter VIII, Section 2]{liggett-1985} (under which 
the bells ring at the same sites at the same types in 
all the coupled single species systems).

More generally, we can project from a ``finer'' multispecies
system to another ``coarser'' one, allowing the merging of two or more adjacent classes into one:

\begin{prop}
\label{prop:colouring}
Let $\phi:\NN\mapsto\NN$ be any function with $\phi(0)=0$ 
which is weakly order-preserving (i.e.\ for all $i<j$, $\phi(i)\leq \phi(j)$).
Then the multispecies PushTASEP with particle content
$\lambda=(\lambda_1,\dots, \lambda_n))$.
lumps to the multispecies PushTASEP
with particle content given by $\phi(\lambda)$, 
where $\phi(\lambda)$ is the partition 
$(\phi(\lambda_1),\dots, \phi(\lambda_n))$.
\end{prop}

\begin{proof}
This is an elementary consequence of the dynamics. 
We ``recolour'' the particles of the system,
giving any particle previously of species $r$ the
new label $\phi(r)$. 
It is an easy case-analysis to check that for any transition
(caused by a bell at some site $j$), the result is independent of whether the recolouring is performed before or after the transition. 
For example, consider the system with content $\langle 0^2, 1^2, 2^2, 3^1, 4^1 \rangle$ and the second transition in \eqref{multilrep-trans-eg}. 
Applying the map $\phi$ given by $\phi(0)=0$, $\phi(1)=\phi(2)=\phi(3)=1$, $\phi(4)=2$, we
recolour to the system with content $\langle 0^2, 1^5, 2^1)$,and the transition becomes $(1, 0, 1, 2, 1, 0, 1, 1) \to (1, 1, 1, 2, 1, 0, 0, 1)$ in either case. 
\end{proof}

To map to a single species system as described just above, we would apply the map $\phi^r$ with $\phi^r(i)=0$ for $i<r$, and $\phi^r(i)=1$ for $i\geq r$.

\section{Interchangeability of rates}

\label{sec:single symm}
Weber \cite{Weber1979} proved an 
\textit{interchangeability} result 
for exponential queueing servers in tandem. 
Consider two independent $./M/1$ queueing servers
(i.e.\ servers who offer service at the times
of a Poisson process) in tandem, with service rates
$\mu_1$ and $\mu_2$. The first queue 
has some arrival process $A$, with an arbitrary distribution (for example, it could be deterministic), which is independent of the service processes. By ``tandem''
we mean that a customer leaving the first queue immediately joins the second queue; the
departure process from the first server is the arrival process of the second.

Then Weber's result is that the law of the departure
process from the system (i.e.\ of the departure process from 
the second queue) is the same if the rates $\mu_1$ and $\mu_2$
are interchanged. 

Various different proofs of this result were subsequently 
given; a significant one from our point of view was a 
coupling proof by Tsoucas and Walrand \cite{TsoucasWalrand1987}.
They constructed a coupling of two pairs of independent Poisson processes, $(S_1, S_2)$ with rates $\mu_1, \mu_2$ 
and $(\tS_1, \tS_2)$ with rates $\mu _2, \mu_1$,
such that for any arrival process $A$,
the output of the system with arrival process $A$ and service 
processes $(S_1, S_2)$ is the same as that of the system with 
arrival process $A$ and serice processes $(\tS_1, \tS_2)$. 

This stronger result, holding simultaneously for all arrival 
processes, makes it possible to extend to a multispecies
framework (since, along the lines we have already seen above, 
a multispecies configuration can be seen as a coupling
of several single species configurations). See for 
example \cite{MartinPrabhakar2010} for extensive discussion
and applications. 

In this section we develop similar ideas in a new context,
to obtain interchangeability of rates of PushTASEP stations. 
We will ultimately be able to apply them to get the result 
of \cref{thm:multipath}.

We consider the PushTASEP on the ring for this work, but the result below applies equally to any one-dimensional lattice (such as a closed or open interval or all of $\ZZ$).
Each site $j$ has an associated parameter $\var_j$. Bells ring independently as Poisson processes at each site, with rate $1/\var_j$ at site $j$. 
Recall that when a bell rings at site $j$ in the single species PushTASEP:
\begin{itemize}
    \item if $j$ is empty, nothing happens.
    \item if $j$ is occupied, then:
    \begin{itemize}
        \item $j$ becomes empty;
        \item the first empty site $k$ to the right of $j$ becomes occupied;
    \end{itemize}
    which we call a \emph{transfer of a particle from $j$ to $k$}.
\end{itemize}
For each lattice bond $(j,j+1)$, we have a 
``flux process'' across the bond; a point process $F_{j,j+1}$ which records the times when there is a transfer of a particle from $j$ or a site to its left to $j+1$ or a site to its right.

Fix a site $j$. Consider the system started from, say, time $0$. 
If we know the flux process $F_{j-1,j}$ for the bond on the left of $j$, the initial occupancy of the site $j$, and the Poisson process of bells at site $j$, then we can obtain the flux process $F_{j, j+1}$ 
for the bond on the right of $j$.

Iterating, this now works for any finite interval
$j, j+1, \dots, j+k$, where $k\geq 0$. Given the flux process $F_{j-1,j}$ at the left end of the interval, the initial configuration inside the interval, and the bell processes inside the interval, one can obtain the flux process $F_{j+k, j+k+1}$ at the right end. 
For the moment we only need the case $k=1$.

\newcommand{\rate}{\rho}

We write $\mathrm{PP}(\rate)$ to denote a Poisson process with rate $\rate$. This is our interchangeability result.

\begin{thm}
\label{thm:ABinterchange}
Consider two neighbouring sites $j$, $j+1$ with parameters $\var_j, \var_{j+1}$.
Regard the output flux process $R=F_{j+1, j+2}$ on 
the time-interval $[0,\infty)$
as a function $R(L, A, B, \eta_{01})$ of:
\begin{itemize}
    \item the input flux process $L=F_{j-1, j}$ on $[0,\infty)$;
    \item the Poisson processes $A$ and $B$ of bells at sites $j$ and $j+1$ respectively;
    \item the time-$0$ occupancies of sites $j$ and $j+1$, denoted by $\eta_{01}=(\eta_j, \eta_{j+1})\in\{0,1\}^2$.
\end{itemize}

Then there exists a coupling of 
\begin{align*}
    (A,B)&\sim \mathrm{PP} \left(\frac{1}{\var_j}\right) \otimes \mathrm{PP} \left(\frac{1}{\var_{j+1}}\right) \\
\intertext{and}
(\tA, \tB)&\sim \mathrm{PP}\left(\frac{1}{\var_{j+1}}\right)
\otimes \mathrm{PP}\left(\frac{1}{\var_j}\right)
\end{align*}
with the property that for every locally finite $L$, and for 
every $\eta_{01} \in \{(0,0), \allowbreak  (1,1),(1,0)\}$,
\[
R(L, A, B, \eta_{01}) = R(L, \tA, \tB, \eta_{01})
\]
with probability $1$. 
\end{thm}

See \cref{fig:interchange} for an illustration.
This theorem says that the bell rates $1/\var_j$ and $1/\var_{j+1}$ at the neighbouring sites $j$ and $j+1$ are interchangeable. 

\begin{figure}[ht]
    \centering
\includegraphics[width=0.7\textwidth]{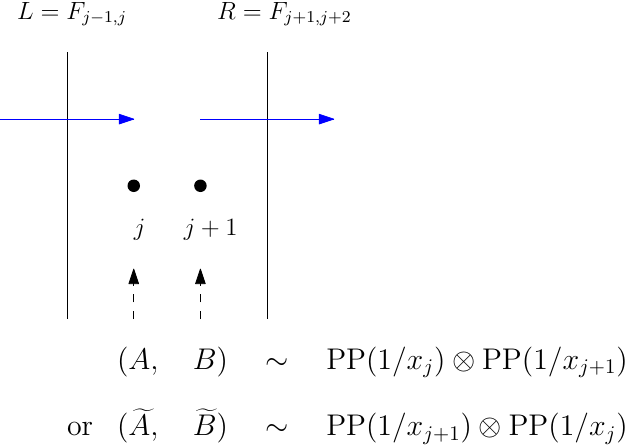}
    \caption{Illustration of the set-up of \cref{thm:ABinterchange}.
    \label{fig:interchange}
    }
\end{figure}

\newcommand{\marked}{w}

Before we begin the proof, we describe the coupling.
First, we will use $\rate = 1/\var_j$ and $\rate'= 1/\var_{j+1}$ in the rest of this section to keep the notation simple.
Recall that a basic fact about Poisson process is the thinning theorem~\cite{durrett-2016}, which in our context reads as follows.
We can describe 
$(A,B)\sim \mathrm{PP}(\rate)\otimes\mathrm{PP}(\rate')$
via a single Poisson process of 
rate $\rate+\rate'$ which is the superposition of $A$ and $B$, with a mark $a$ or $b$ attached to each point, according to whether it comes from the process $A$ or $B$ respectively. 
Given the points of $A+B$, the marks are i.i.d.\ and each is $a$ with probability $\rate/(\rate+\rate')$ and $b$ with probability $\rate'/(\rate+\rate')$. If we index the points of $A+B$ with $\ZZ$ in increasing order, say with point $0$ the last to occur before time $0$ and point $1$ the first to occur after time $0$, then we can identify the sequence of marks with an element
$(\marked_i, i\in \ZZ)\in \{a,b\}^\ZZ$.

Under our coupling, the superposition $\tA+\tB$ will be the same as $A+B$; we will just change (some of) the marks. 
Consider occurrences of the motif 
$ab$, that is, look for $i$ such that $\marked_i=a$, $\marked_{i+1}=b$. 
The set $\cC$ of such $i$ has a distribution which is invariant under interchanging $\rate$ and $\rate'$. 
This is because if $i_1<i_2<\dots<i_n$, 
then $P(i_1,\dots,i_n\in\cC)$
is equal to $0$ if $i_k-i_{k-1}=1$
for some $k$, and otherwise is equal to 
$\rate^n (\rate')^n/(\rate+\rate')^{2n}$.
This means there exists a coupling of $(A,B)$ and $(\tA,\tB)$ with the required distributions such that:
\begin{itemize}
    \item the superpositions $A+B$ and $\tA+\tB$ are the same;
\item
the occurrences of the motif $ab$ are the same in both processes.
\end{itemize}

These are the only properties we will need, but we can make the coupling explicit as follows.
The $ab$ motifs are separated by words of the form $b^* a^*$,
consisting of some number (maybe $0$) of $b$'s followed by some number (maybe $0$) of $a$'s.
Conditional on the length $n$ of the string, the probability that
it consists of $n_1$ $b$'s followed by $n_2$ $a$'s 
(where $0\leq n_1, n_2$, with $n_1+n_2=n$)
is proportional to
\begin{equation}\label{stringprob}
(\rate')^{n_1}\rate^{n_2}.
\end{equation}
An explicit way to realise
the coupling is to 
obtain the process $(\tA,\tB)$
from $(A,B)$ as follows. 
Leave the $ab$ motifs unchanged. As for the $b^{n_b} a^{n_a}$ separating them,
replace it instead by $b^{n_a} a^{n_b}$. 
This is illustrated in \cref{fig:ABcoupling}.
Because of \eqref{stringprob}, this 
has the effect of interchanging the probability of $a$ and $b$ as desired.

\begin{rem}
Equivalently, rather than using this deterministic scheme, one could just resample each of the strings independently according to the new desired measure.
One can also think about all of this in terms of run lengths of consecutive $a$'s and $b$'s (which are independent, geometric ($\rate'/(\rate+\rate')$) for the $a$'s and geometric($\rate/(\rate+\rate')$) for the $b$'s).
\end{rem}

\begin{center}
\begin{figure}[ht]
\includegraphics[width=\textwidth]{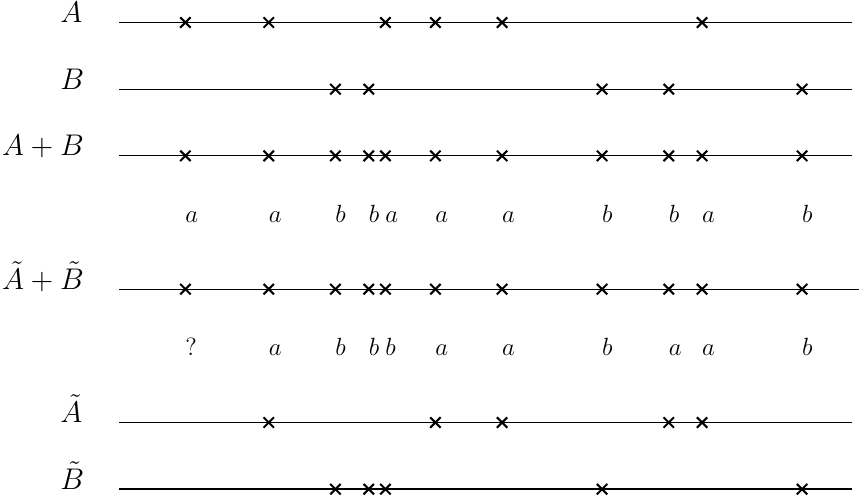}
\caption{
\label{fig:ABcoupling}
Illustration of the coupling scheme. This version shows the deterministic scheme explained above. The point marked with a question mark symbol will either belong to $\tA$ or $\tB$ depending on the points before it.
}
\end{figure}
\end{center}

\begin{proof}[Proof of \cref{thm:ABinterchange}]
We will consider two systems -- call them $S$ and $\tS$. In each one we have two sites $j$ and $j+1$. We have the same input flux process $L=F_{j-1,j}$. We have bell processes at $j$ and $j+1$ respectively given by $A$ and $B$ in system $S$, and $\tA$ and $\tB$ in system $\tS$. 
The claim is that the output flux processes $F_{j+1,j+2}$ are exactly the same (not just the same in distribution) in both systems. 

The initial configuration $\eta$ is the same in both systems, 
with $\eta_{01} = (\eta_j, \eta_{j+1})$ being one of $(0,0)$, $(1,0)$, or $(1,1)$. See \cref{rem:eta=01} after this proof for discussion of why $(\eta_j, \eta_{j+1})$ cannot be $(0,1)$.
Since the flux into $\{j,j+1\}$ on the left is the same in the two systems, and they start with the same number of particles, in order to conclude that the flux out of $\{j,j+1\}$ on the right is the same in the two systems it will be enough to observe that the number of particles in these two sites remains the same across both systems for all times. 

We work by induction to determine whether there could be a first moment when the number of particles in $S$ and $\tS$ do not agree. Let us first check whether an extra particle could be added in one system and not in the other.
This could happen only when a point of $L$ occurs. But a point of $L$ always adds a particle unless the system is occupied in both $j$ and $j+1$. So if a point of $L$ adds a particle in one system but not the other, it can only be because the numbers already didn't agree. 

Hence it would have to be that one system loses a particle, while the other does not. This would have to happen due to a bell ringing at either $j$ or $j+1$. Recall that such bells happen simultaneously in $S$ and $\tS$, 
though not necessarily in the same place (that is, the superpositions $A+B$ and $\tA+\tB$ are the same).
First suppose both systems were previously full. Then any such bell loses a particle from \textit{both} systems. So, the only way for the numbers to become unbalanced would be for both systems to contain one particle, and then for one of them (but not the other) to become empty. We will now show that both systems necessarily become empty at the same time.

This boils down to a case analysis.
First, take the case where at time $0$ these two sites are full (i.e. $(\eta_j, \eta_{j+1}) = (1,1)$), or in the state $(1,0)$. In this case, for the system to be empty at a given time, it is necessary and sufficient that since the last point of $L$ (or since time $0$ if there has been no point of $L$), there has been a bell at $j$ followed by a bell at $j+1$. That is, we need to have observed the motif $ab$ in the marks of the Poisson process. But our coupling ensures that this motif occurs at identical times in the systems $S$ and $\tS$. 
The case where the system starts empty (i.e. $\eta=(0,0)$) is similar. In this case, if there has been no point of $L$, then the system remains empty. If there has been a point of $L$, then as above, for the system to be empty we need to have seen the motif $ab$ since the last such point. 

Hence the number of particles in the two systems $S$ and $\tS$ must remain the same at all times. It follows that the output flux processes $R$ and $\tR$ are the same, as desired. 
\end{proof}

\begin{rem}
\label{rem:eta=01}
The proof above fails, as it must, if instead we start from the configuration $(0,1)$. Then if we observe a bell at $j+1$ in one system, but at $j$ in the other, then the former may have a point of the output flux process, while the latter does not.
\end{rem}

\begin{rem}
In the setting above we can also take the ``input
process'' to be defined on all times in $\RR$, 
and similarly obtain an output process for all times in $\RR$.
(This works straightforwardly since there are arbitrarily early 
times when we know the occupancy state of $j$;
whenever the bell at $j$ rings, it becomes empty.)
Then we may regard a single PushTASEP site as an operator
which maps distributions of input processes to 
distributions of output processes. Recalling
Burke's theorem for exponential-server queues, we
can ask whether all the ergodic fixed points of this map to 
be Poisson processes.
\end{rem}

The fact that our coupling of $(A,B)$ and $(\tA, \tB)$ works
simultaneously over all input processes $L$ means that we can 
pass from a single species result to a multispecies result. 

\begin{proof}[Proof of \cref{thm:multipath}]
First consider a single species PushTASEP on a ring of $n$ sites, with $k$ particles for some $1\leq k\leq n-1$. 

Imagine two versions $S$ and $\tS$ with the same initial configuration of this system, with the rates $1/\var_j$ and $1/\var_{j+1}$ at sites $j$ and $j+1$ swapped between the two. We may couple $S$ and $\tS$ using identical bell processes between the two systems everywhere except $j$ and $j+1$, and using the coupling provided by \cref{thm:ABinterchange} for the bell processes at sites $j$ and $j+1$. 

Suppose we start the two systems in the same configuration at time $0$, with the configuration at sites $(j,j+1)$ being one of $(0,0)$, $(1,0)$ or $(1,1)$. Imagining the two systems evolving one bell at a time, we can see there is never a first moment when the flux processes $F_{j-1,j}$ and $F_{j+1, j+2}$ are different between the two systems. One minor subtlety can arise on the ring if there are $n-1$ particles. A bell at site $j+1$ can ``cause" a push which propagates almost all the way around the ring, to create a point in the flux process $F_{j-1,j}$. For this to happen, both systems must be in state $(0,1)$ on the sites $(j,j+1)$, and both must experience a bell at $j+1$. Afterwards they are both in state $(1,0)$ on those two sites, and the coupling between the two systems proceeds without problem. 

Now we proceed to the \textit{multispecies} PushTASEP on the ring of size $n$. 
We view the multispecies PushTASEP as a coupling of several single species PushTASEPs. 
We take some initial condition on the multispecies system in which the particle at site $j$ is at least as strong as the particle at site $j+1$. This means that in all the single species projections, the configuration at $(j,j+1)$ is in $\{(0,0), (1,0), (1,1)\}$. 
Again we consider a coupling between two systems with rates $(1/\var_j,1/\var_{j+1})$ or $(1/\var_{j+1},1/\var_j)$ at $j$ and $j+1$ respectively, with all the other rates remaining the same. 
As before we use the coupling of the bell processes at sites $j$ and $j+1$ given by \cref{thm:ABinterchange}, and at all other sites, we keep the bell processes identical between the two systems. Since all 
the single species projections look the same between the two systems outside $\{j,j+1\}$, the same
is true of the multispecies system which is the coupling of those single species systems. 

Now as in the statement of \cref{thm:multipath},
we may fix any $k$, and consider an initial condition in which 
$\eta_{k+1}\geq \dots\geq \eta_n$. From the argument above,
we may exchange the parameters of any two neighbouring sites
in $\{k+1,\dots, n\}$ while preserving the distribution of the
process as observed on sites $\{1,2,\dots, k\}$. But then 
we can perform a sequence of such nearest-neighbour transpositions
to realise any desired permutation of the parameters $\var_{k+1}, 
\dots, \var_n$. So indeed the distribution is symmetric
under permutation of these parameters, as desired. 

Finally we wish to show the same property for events 
defined on the time interval $[0,\infty)$ starting from the 
stationary distribution. Consider the system started from any initial 
configuration satisfying $\eta_{k+1}\geq \dots\geq \eta_n$.
Since the system is an irreducible Markov chain on a finite
state-space, it has a unique stationary distribution, to which
it converges from this initial condition. 
In particular, we may arbitrarily closely approximate the probability of any event on $[0,\infty)$ 
starting from stationarity by taking the probability of 
an appropriate event on $[T, \infty)$ for large enough $T$. 
Since the probabilities of all such approximating events 
are symmetric in $\var_{k+1}, \dots, \var_n$, the same 
must also be true of the original event. 
\end{proof}

\section{Multiline PushTASEP}
\label{sec:multiline}
In this section we discuss how the strategy of Ferrari and Martin \cite{ferrari-martin-2006} adapts
to give the result of \cref{thm:multilrep-statprob}. 
Theorem 4 of \cite{ferrari-martin-2006} refers
to a multispecies version of the Hammersley-Aldous-Diaconis process
-- see the comment at the end of Section 2 of that paper 
for the relationship between the HAD and the PushTASEP (or 
``long-range exclusion process'') via reversing the order of the particles. The main novelty in our setting is the introduction
of site-wise inhomogeneity via the parameters $\var_1, \dots, \var_n$. 

Here is the outline of the strategy for studying
the stationary distribution of the multispecies process 
on $\Omega_\lambda$, where the partition $\lambda$ gives
the contents on the system. 
\begin{itemize}
\item
We consider a set $\hOmega_\lambda$ of \textit{multiline diagrams}
and a map $\Pi$ from $\hOmega_\lambda$ to the set $\Omega_\lambda$
of PushTASEP configurations. 
\item
We construct a Markov chain on the set of multiline diagrams (the multiline process) and find its stationary distribution using a time-reversal argument
(\cref{thm:stat dist multiline}).
In the homogeneous case, this stationary distribution was simply the uniform distribution; now that we
introduce site-wise inhomogeneity, we get a distribution 
with weights proportional to monomials in the parameters $\var_1, \dots, \var_n$. 
\item 
We show that the projection of the multiline process under
the map $\Pi$ is the multispecies PushTASEP 
(\cref{lem:projection}). 
From this we deduce the stationary distribution $\pi$ of the
multispecies PushTASEP, with weights proportional to sums 
of monomials in the parameters $\var_1, \dots, \var_n$
(\cref{thm:stat prob sum}).
\end{itemize}
To obtain the form of the stationary distribution given in 
\cref{thm:multilrep-statprob}, we finally apply the
correspondence between ASEP polynomials and weights of multiline
diagrams given by \cite{corteel-mandelshtam-williams-2022}. 

We concentrate particularly on
the argument for \cref{thm:stat dist multiline}, where the inhomgoneity plays a significant role. In contrast, in the argument
for \cref{lem:projection}, the rates are irrelevant, 
and we outline the idea briefly. 

Before we move on to the proof, we state important properties of the multispecies PushTASEP. First, one can see that it is irreducible if all $\var_i$'s are positive and finite by the following argument. Suppose $\eta, \eta' \in \Omega_\lambda$. Starting from $\eta$, initiate transitions starting at the locations of species $s$'s until they are in their correct positions in $\eta'$. Now, initiate transitions starting at the $(s-1)$'s. These will clearly not affect the $s$'s. Continue this way until all particles are in their correct location in $\eta'$.

\subsection{Multiline diagrams}

Multiline diagrams were introduced
to construct the stationary distribution of the multispecies 
TASEP~\cite{ferrari-martin-2007}. The basic definition has since been extended in several ways.
One rather general set-up is given by \cite{corteel-mandelshtam-williams-2022}, incorporating parameters $t\geq 0$, 
$q\geq 0$ and variables $x_1, \dots, x_n$ to give combinatorial
constructions of the ASEP polynomials and non-symmetric Macdonald polynomials. Here we need only the basic case $t=0$ and $q=1$; 
the variables here denoted $\var_1, \dots, \var_n$ 
reflect the site-wise inhomogeneity. 

A \textit{multiline diagram} 
with contents given by $\lambda = \langle 0^{m_0}, 1^{m_1},  \dots,s^{m_s} \rangle$, 
is a configuration on a discrete cylinder with $s$ rows and 
$n = \sum_i m_i$ columns as follows. Rows are indexed $1$ to $s$ from bottom to top, and columns $1$ to $n$ from left 
to right. Each site has either a particle (denoted $\occ$)
or a vacancy (denoted $\vac$). 
In row $i$, there are $M_i = m_{i} + \cdots + m_s$ particles,
and so $n-M_i$ vacancies. 
The set of such multiline diagrams is 
denoted $\widehat{\Omega}_\lambda$.

For example, 
\begin{multline}
\label{multiline-diag-eg-111}
\widehat\Omega_{(2,1,0)} = 
\left\{ 
\begin{array}{ccc}
\occ & \vac & \vac \\
\occ & \occ & \vac
\end{array},
\begin{array}{ccc}
\occ & \vac & \vac \\
\occ & \vac & \occ
\end{array},
\begin{array}{ccc}
\occ & \vac & \vac \\
\vac & \occ & \occ
\end{array},
\begin{array}{ccc}
\vac & \occ & \vac \\
\occ & \occ & \vac
\end{array}, \right. \\
\left. 
\begin{array}{ccc}
\vac & \occ & \vac \\
\occ & \vac & \occ
\end{array}, 
\begin{array}{ccc}
\vac & \occ & \vac \\
\vac & \occ & \occ
\end{array},
\begin{array}{ccc}
\vac & \vac & \occ \\
\occ & \occ & \vac
\end{array},
\begin{array}{ccc}
\vac & \vac & \occ \\
\occ & \vac & \occ
\end{array},
\begin{array}{ccc}
\vac & \vac & \occ \\
\vac & \occ & \occ
\end{array}
\right\}
\end{multline}

We now describe a map 
$\Pi: \widehat{\Omega}_\lambda \to \Omega_\lambda$ from the set of 
multiline diagrams to the set $\Omega_\lambda$ of 
multispecies PushTASEP configurations with contents
given by the same partition $\lambda$. 

\begin{center}
\begin{figure}[h!]
\begin{tikzpicture}[scale=1]
\matrix [column sep=0.5cm, row sep = 0.1cm, ampersand replacement=\&] 
{
\node {$\vac$}; \& \node{$\vac$}; \& \node(a1) {$\occ_{4}$}; \& \node {$\vac$}; \& \node {$\vac$}; \& \node {$\vac$}; \& \node(a21) {$\occ_4$}; \& \node {$\vac$};  \& \node {$\vac$}; \\
\node {$\vac$}; \& \node {$\vac$}; \& \node(a2) {$\vac$}; \& \node {$\vac$}; \& \node {$\vac$}; \& \node(a3) {$\occ_4$}; \& \node(a22) {$\occ_4$}; \& \node(b1) {$\occ_3$}; \& \node {$\vac$}; \\
\node(b3) {$\occ_3$}; \& \node(c1){$\occ_2$}; \& \node {$\vac$}; \& \node(c21) {$\occ_2$};\& \node {$\vac$}; \& \node {$\occ_4$};\& \node(a23) {$\vac$}; \& \node(b2) {$\vac$}; \& \node(a24) {$\occ_4$};\\ 
\node(a26) {$\occ_4$}; \& \node(b4){$\occ_3$}; \& \node {$\vac_0$}; \& \node(c3) {$\occ_2$}; \&\node(c23) {$\occ_2$}; \&\node (a4){$\occ_4$};\& \node(d1) {$\occ_1$}; \& \node(d21) {$\occ_1$}; \& \node(a25) {$\vac_0$};\\
};

\draw [-,thick,red] (a1.center) -- (a3.center) -- (a4.center);
\draw[-,thick,red] (a21.center) -- (a22.center) 
-- (a24.center) --(5.2,-0.5); 
\draw [-,thick, red] (-5.2,-0.5) -- (a26.center);
\draw [-,thick, blue] (b1.center) -- (5.2, 0);
\draw [-,thick, blue] (-5.2, 0) -- (b3.center)  -- (b4.center);
\draw [-,thick,green] (c1.center) 
-- (c3.center) ;
\draw [-,thick,green] (c21.center) 
-- (c23.center);
\end{tikzpicture}

\caption{A multiline diagram $\hat{\eta} \in \widehat{\Omega}_\lambda$ with $\lambda = (4, 4, 3, 2, 2, 1, 1, 0, 0)$ so that $n=9$, with 
the bully-path projection to a configuration of the multispecies TASEP. 
Here $\Pi(\eta)$ is the configuration $(4,3,0,2,2,4,1,1,0)
\in\Omega_\lambda$, as given by the bottom row.}
\label{fig:bullypatheg}
\end{figure}
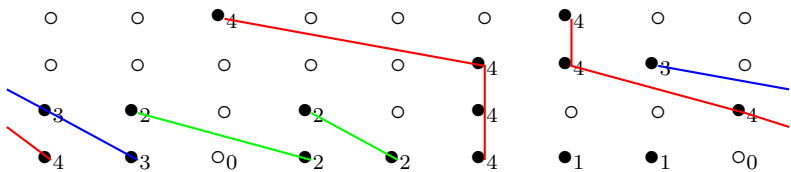
\end{center}

Suppose $\hat{\eta} \in \widehat{\Omega}_\lambda$. 
To obtain $\Pi(\heta)$ from $\heta$ we will assign each of the 
particles in the multiline diagram $\heta$ a species label
from $1$ to $s$. Any particle in row $r$ receives a label at 
least as large as $r$.
The assignment can be done recursively row by row starting at the top:
\begin{enumerate}
\item \label{it:1}
All the particles in row $s$ are labelled $s$.
\item 
Suppose we have labelled all the particles 
in rows $s, s-1,\dots, r+1$. To label the particles in row $r$, 
we take in turn each of the particles in row $r+1$, 
in decreasing order of labels. We break ties arbitrarily, 
for example from left to right. To each of those row-$(r+1)$ 
particles, we are going to match a row-$r$ particle.
\item
To match a row-$(r+1)$ particle, say in column $j$: we look 
in columns $j, j+1, \dots$ in turn (wrapping cyclically around the ring from $n$ to $1$ if needed), until we first find a particle in row $r$ which has not yet been matched. The first one we find is the match of the row-$(r+1)$ particle, and it receives the same species. 
\item 
To all the row-$r$ particles which remain unmatched once
all the row-$(r+1)$ particles have been considered, we assign
species $r$. 
\item
In this way we recursively label all the particles in the 
diagram. Finally we assign label $0$ to all the 
vacancies in row $1$. The bottom row can now be read as 
a configuration in $\Omega_\lambda$, and this is $\Pi(\heta)$. 
\end{enumerate}
An example of this procedure is given in \cref{fig:bullypatheg}.

When $\Pi(\heta)=\eta$, we may say that the multiline diagram
$\heta$ has bottom row $\eta$. 

This is exactly the same procedure as was done for the TASEP
on the ring in \cite{ferrari-martin-2007} (with the 
notational difference that the numbering of the particles is
reversed).

\subsection{Multiline process}
We now construct a process on multiline diagrams,
i.e.\ a Markov chain whose state-space is
$\hOmega_\lambda$.

As in the PushTASEP we will associate a bell which rings at
rate $1/\var_i$ to each site $i$. We think of such a bell ringing
at the corresponding site in the bottom row, i.e. at $(1,i_i)$
where $i_1=i$, and it will cause a PushTASEP jump in the (single-species) particle configuration on the bottom row. Then in turn it will trigger
a chain reaction of bells at sites $(2, i_2),\dots, (s, i_s)$
on each row of the diagram which cause PushTASEP jumps in the 
configuration of that row. 
This occurs in the following way. When the bell 
on row $r$ rings at $(r, i_r)$ for some $1\leq r\leq s$, either
\begin{itemize}
\item[(a)]
the site $(r, i_r)$ is empty, in which case the configuration on row $r$ remains unchanged, and we set $i_{r+1}=i_r$: the new bell rings immediately above; 
\item[(b)]
the site $(r, i_r)$ contains a particle. This particle
jumps to the first empty site on the same row to its right (wrapping cyclically). Let $i_{r+1}$ be the column to which it jumps, so that the new bell rings above the destination site of the jumping particle.
\end{itemize}
These PushTASEP moves in all rows occur simultaneously.
We say the transition starts in column $i_1$ and ends in column $i_{s+1}$ (the destination site of the particle jumping in row $s$). 

We call this chain
the {\em multiline PushTASEP}. See \cref{fig:multiline transition} for an example of a transition.

\begin{center}
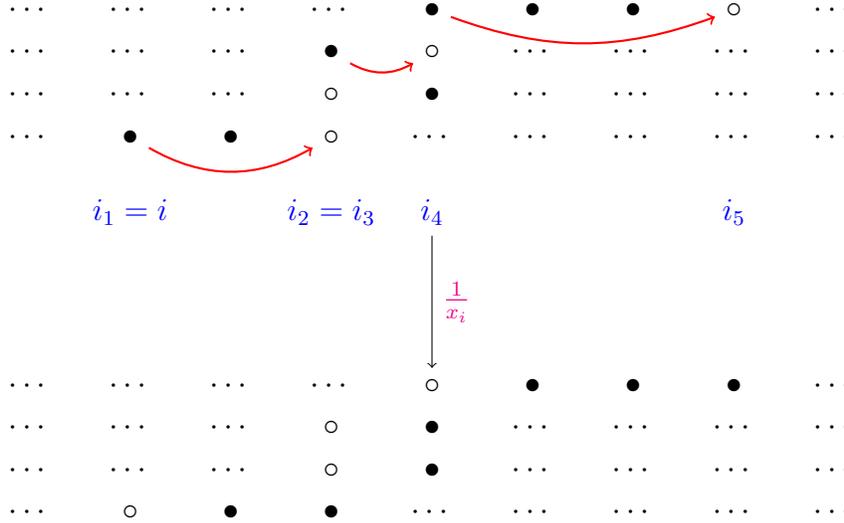
\begin{figure}[h!]
\begin{tikzpicture}[scale=1]
\matrix [column sep=0.5cm, row sep = 0.1cm, ampersand replacement=\&] 
{
\node {$\cdots$}; \& \node {$\cdots$}; \& \node {$\cdots$}; \& \node {$\cdots$}; \& \node (a1) {$\occ$}; \& \node {$\occ$}; \& \node {$\occ$}; \& \node(a2) {$\vac$};  \& \node {$\cdots$}; \\
\node {$\cdots$}; \& \node {$\cdots$}; \& \node {$\cdots$}; \& \node (b1) {$\occ$}; \& \node (b2) {$\vac$}; \& \node {$\cdots$}; \& \node {$\cdots$}; \& \node {$\cdots$}; \& \node {$\cdots$}; \\
\node {$\cdots$}; \& \node{$\cdots$}; \& \node {$\cdots$}; \& \node {$\vac$};\& \node {$\occ$}; \& \node {$\cdots$};\& \node {$\cdots$}; \& \node {$\cdots$}; \& \node {$\cdots$};\\ 
\node {$\cdots$}; \& \node (c1) {$\occ$}; \& \node {$\occ$}; \& \node (c2) {$\vac$}; \&\node (d1) {$\cdots$}; \&\node {$\cdots$};\& \node {$\cdots$}; \& \node (e1) {$\cdots$}; \& \node {$\cdots$};\\
};

\node [below of = c1] {\tcb{$i_1 = i$}};
\node [below of = c2] {\tcb{$i_2 = i_3$}};
\node (dd1) [below of = d1] {\tcb{$i_4$}};
\node [below of = e1] {\tcb{$i_5$}};

\draw [->, thick, red] (c1) to [out = -30, in = -150] (c2);

\draw [->, thick, red] (b1) to [out = -30, in = -150] (b2);

\draw [->, thick, red] (a1) to [out = -20, in = -160] (a2);

\matrix [column sep=0.5cm, row sep = 0.1cm, ampersand replacement=\&] at (0,-5)
{
\node {$\cdots$}; \& \node {$\cdots$}; \& \node {$\cdots$}; \& \node  {$\cdots$}; \& \node (a3) {$\vac$}; \& \node {$\occ$}; \& \node {$\occ$}; \& \node(a4) {$\occ$};  \& \node {$\cdots$}; \\
\node {$\cdots$}; \& \node {$\cdots$}; \& \node {$\cdots$}; \& \node (b3) {$\vac$}; \& \node (b4) {$\occ$}; \& \node {$\cdots$}; \& \node {$\cdots$}; \& \node {$\cdots$}; \& \node {$\cdots$}; \\
\node {$\cdots$}; \& \node{$\cdots$}; \& \node {$\cdots$}; \& \node {$\vac$};\& \node {$\occ$}; \& \node {$\cdots$};\& \node {$\cdots$}; \& \node {$\cdots$}; \& \node {$\cdots$};\\ 
\node {$\cdots$}; \& \node (c3) {$\vac$}; \& \node {$\occ$}; \& \node (c4) {$\occ$}; \&\node {$\cdots$}; \&\node {$\cdots$};\& \node {$\cdots$}; \& \node {$\cdots$}; \& \node {$\cdots$};\\
};

\draw [->] (dd1.south) to node [midway, right, color = magenta] {$\frac 1{\var_i}$} (a3);

\end{tikzpicture}
\caption{An illustration of a transition in the multiline PushTASEP.}
\label{fig:multiline transition}
\end{figure}
\end{center}

\subsection{Stationary distribution}

We will now derive the stationary distribution for the multiline PushTASEP.
Let $\hat{\eta} \in \widehat{\Omega}_\lambda$.
Denote by $v_i(\hat{\eta})$ the number of $\occ$'s in the $i$'th column of $\hat{\eta}$.
The \emph{weight} of $\hat{\eta}$ is defined as
\begin{equation}
\label{wt_x}
\wt(\hat{\eta}) = \prod_{i = 1}^n \var_i^{v_i(\hat{\eta})}.
\end{equation}
For the example in \cref{fig:bullypatheg}, the weight is
$\var_1^2 \var_2^2 \var_3 \var_4^2 \var_5 \var_6^3 \var_7^3 \var_8^2 \var_9$.

\begin{thm}
\label{thm:stat dist multiline}
The distribution $\hat{\pi}$ on $\hOmega_\lambda$ defined by 
\begin{equation}\label{pihatdef}
\hat{\pi} (\hat{\eta}) = \frac{\wt(\hat{\eta})}{Z_\lambda}  
= \frac{1}{Z_\lambda} \prod_{i=1}^n \var_i^{v_i(\hat{\eta})},
\end{equation}
where $Z_\lambda=\sum_{\heta'\in\hOmega_\lambda} \wt(\heta')$,
is stationary for the multiline process with content $\lambda$.
\end{thm}

Since the number of particles in row $r$ is fixed to be $M_r$,
we may factorise the partition function $Z_\lambda$ into 
a product of terms corresponding to the different rows, to obtain
\begin{align}\nonumber
Z_\lambda&=
\prod_{r=1}^s
\sum_{1\leq i_i < \dots < i_{M_r}\leq n} \var_{i_1}\dots\var_{i_{M_r}}
\\
\label{Zlambda}
&=\prod_{r=1}^s e_{M_r}(\var_1,\dots, \var_n).
\end{align}

To prove \cref{thm:stat dist multiline} 
we explicitly identify the time-reversal of the Markov chain in stationarity. The key lemma is the following useful result, which we know of through Tom Liggett and Pablo Ferrari (but we do not know a precise reference). 

\begin{lem}
\label{lem:pairwise}
Consider a continuous-time Markov chain on 
a state-space $S$ with rates $R = (r(s,s') \mid s, s' \in S, s \neq s')$. 

Suppose $\pi$ is a probability distribution on $S$, and let
$R^* = (r^*(s,s') \mid s, s' \in S, s \neq s')$ be 
another collection of rates, such that
\begin{align}
\label{eqwts}
\pi(s) r^*(s,s') =& \pi(s') r(s',s) \quad \forall s, s'\in S, s \neq s', \\
\label{eqsum}
\sum_{\substack{s' \in S \\ s \neq s'}} r(s,s') =& 
\sum_{\substack{s' \in S \\ s \neq s'}} r^*(s,s'), \quad \forall s \in S.
\end{align}
Then $\pi$ is a stationary distribution for the chain. 
\end{lem}

\begin{proof}
The point is that the rates $R^*$ are those for the time-reversal
of the chain in its stationary distribution $\pi$. 

It's straightforward to verify the master equation. The total incoming weight into state $s$ in distribution $\pi$ is
\[
\sum_{s' \in S} \pi(s') r(s', s) 
= \sum_{s' \in S} r^*(s', s) \pi(s) 
= \pi(s) \sum_{s' \in S} r(s, s'), 
\]
where we have used \eqref{eqwts} for the first equality
and \eqref{eqsum} for the second.
The right hand side is of course the total outgoing weight from $s$. 
\end{proof}

We now define the alternative dynamics on the set $\widehat{\Omega}_\lambda$ of multiline diagrams,
which will play the role of the time-reversed rates $R^*$ in 
\cref{lem:pairwise}. The transitions are similar to before,
but reversed both vertically and horizontally. Each row
of the diagram sees a single species PushTASEP jump from right to 
left. As before, a bell rings in column $j$ 
with rate $1/\var_j$. Now we associate that bell to the top row,
i.e.\ to the site $(s, j_{s+1})$ where $j_{s+1}:=j$. If
there is no particle there, the top row sees no change, and 
we set $j_{s}=j_{s+1}$. Otherwise, the particle there jumps to 
the nearest available site \emph{to its left}. 
(As usual, we wrap cyclically from site $1$ to site $n$ as needed.)
We call $j_{s}$ the column 
it jumps to, and generate a bell at $(s-1, j_{s})$. 
Continue this way until we reach row $1$. 
As before, all these changes happen simultaneously.
These transitions define the {\em reverse multiline PushTASEP} on $\widehat{\Omega}_\lambda$. See \cref{fig:rev multiline transition} for an example of a transition.

\begin{center}
\begin{figure}[h!]
\begin{tikzpicture}[scale=1]
\matrix [column sep=0.5cm, row sep = 0.1cm, ampersand replacement=\&] 
{
\node {$\cdots$}; \& \node (a-1) {$\cdots$}; \& \node {$\cdots$}; \& \node (a0) {$\cdots$}; \& \node (a1) {$\vac$}; \& \node {$\occ$}; \& \node {$\occ$}; \& \node(a2) {$\occ$};  \& \node {$\cdots$}; \\
\node {$\cdots$}; \& \node {$\cdots$}; \& \node {$\cdots$}; \& \node (b1) {$\vac$}; \& \node (b2) {$\occ$}; \& \node {$\cdots$}; \& \node {$\cdots$}; \& \node {$\cdots$}; \& \node {$\cdots$}; \\
\node {$\cdots$}; \& \node{$\cdots$}; \& \node {$\cdots$}; \& \node {$\vac$};\& \node {$\occ$}; \& \node {$\cdots$};\& \node {$\cdots$}; \& \node {$\cdots$}; \& \node {$\cdots$};\\ 
\node {$\cdots$}; \& \node (c1) {$\vac$}; \& \node {$\occ$}; \& \node (c2) {$\occ$}; \&\node (d1){$\cdots$}; \&\node {$\cdots$};\& \node {$\cdots$}; \& \node {$\cdots$}; \& \node {$\cdots$};\\
};

\node [above of = a2] {\tcb{$j_5 = j$}};
\node [above of = a1] {\tcb{$j_4$}};
\node [above of = a0] {\tcb{$j_3 = j_2$}};
\node [above of = a-1] {\tcb{$j_1$}};

\draw [<-, thick, red] (c1) to [out = -30, in = -150] (c2);

\draw [<-, thick, red] (b1) to [out = -30, in = -150] (b2);

\draw [<-, thick, red] (a1) to [out = -20, in = -160] (a2);

\matrix [column sep=0.5cm, row sep = 0.1cm, ampersand replacement=\&] at (0,-4)
{
\node {$\cdots$}; \& \node {$\cdots$}; \& \node {$\cdots$}; \& \node {$\cdots$}; \& \node (a3) {$\occ$}; \& \node {$\occ$}; \& \node {$\occ$}; \& \node(a4) {$\vac$};  \& \node {$\cdots$}; \\
\node {$\cdots$}; \& \node {$\cdots$}; \& \node {$\cdots$}; \& \node (b3) {$\occ$}; \& \node (b4) {$\vac$}; \& \node {$\cdots$}; \& \node {$\cdots$}; \& \node {$\cdots$}; \& \node {$\cdots$}; \\
\node {$\cdots$}; \& \node{$\cdots$}; \& \node {$\cdots$}; \& \node {$\vac$};\& \node {$\occ$}; \& \node {$\cdots$};\& \node {$\cdots$}; \& \node {$\cdots$}; \& \node {$\cdots$};\\ 
\node {$\cdots$}; \& \node (c3) {$\occ$}; \& \node {$\occ$}; \& \node (c4) {$\vac$}; \&\node {$\cdots$}; \&\node {$\cdots$};\& \node {$\cdots$}; \& \node {$\cdots$}; \& \node {$\cdots$};\\
};

\draw [->] (d1) to node [midway, right, color = magenta] {$\frac 1{\var_j}$} (a3);

\end{tikzpicture}
\caption{An illustration of a transition in the reverse multiline PushTASEP.}
\label{fig:rev multiline transition}
\end{figure}
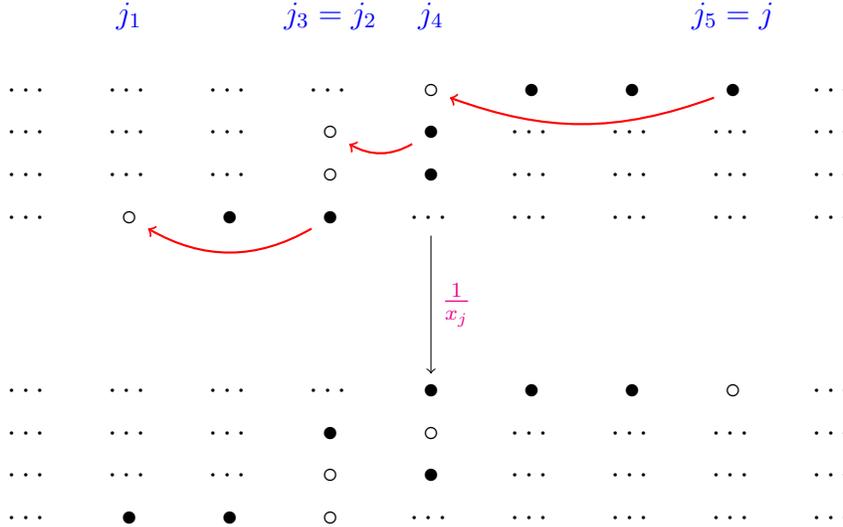
\end{center}

\begin{lem}
\label{lem:forward reverse}
Let $\hat{\eta},\hat{\eta}' \in \widehat{\Omega}_\lambda$.
If there is a transition from $\hat{\eta}$ to $\hat{\eta}'$ in the multiline PushTASEP with rate $1/\var_i$ ending at site $(s,j)$, then there is a transition from $\hat{\eta}'$ to $\hat{\eta}$ in the reverse multiline PushTASEP with rate $1/\var_j$, and we have $\hat{\pi} (\hat{\eta})/\var_i  = \hat{\pi}(\hat{\eta}')/\var_j$,
where $\hat{\pi}$ is defined by \eqref{pihatdef}.
Further, the 
rate of jumping away from $\hat{\eta}$ is the same in both processes.
\end{lem}

\begin{proof}
It might be useful to compare the transitions in \cref{fig:multiline transition,fig:rev multiline transition}, which are reverses of each other, to follow the proof. 

First of all, the number of $\vac$'s in all columns except $i$ and $j$ are the same in both $\hat{\eta}$ and $\hat{\eta}'$. Also, $\hat{\eta}$ has one more (resp. one fewer) $\occ$ in column $i$ (resp. column $j$) compared to $\hat{\eta}'$. Therefore, $\hat{\pi} (\hat{\eta})/\hat{\pi}(\hat{\eta}') = \var_i/\var_j$, proving the first claim.

Now, focus on the total outgoing rate of transitions from $\hat{\eta}$ in the multiline PushTASEP. For each $i$, there is a transition with rate $1/\var_i$ except when there is no particle, i.e. $\occ$, in column $i$. Exactly the same holds for the reverse process. Therefore, the total outgoing rate is the same, proving the second claim.
\end{proof}

\begin{proof}[Proof of \cref{thm:stat dist multiline}]
\cref{lem:forward reverse} now tells us that the conditions
of \cref{lem:pairwise} hold for the distribution 
$\hat{\pi}$ and for the rates of the forward and reverse
multiline processes. Hence indeed $\hat\pi$ is stationary. 
\end{proof}

\subsection{Projection to the multispecies PushTASEP}

\begin{lem}\label{lem:projection}
Consider the multiline process in stationarity, say 
$(\heta_u, u\geq 0)$ where $u$ is the time parameter. 
Then the projection $(\Pi(\heta_u), u\geq 0)$ 
is the multispecies PushTASEP process in stationarity.  
\end{lem}
We have constructed both the PushTASEP
and the multiline process as processes governed by
independent Poisson processes of bells with rate
$1/\var_j$ at site $j$, for $1\leq j\leq n$. 
To prove \cref{lem:projection}, it then suffices to check the following property. 

\begin{lem}\label{lem:jumpj}
Suppose $\heta\in\hOmega_\lambda$ and 
$\eta=\Pi(\heta)\in \Omega_\lambda$.
Let $\heta^j$ be the state resulting from $\heta$ when a 
bell rings at site $j$ in the multiline process. 
Let $\eta^j$ be the state resulting from $\eta$ when a bell
rings at site $j$ in the multispecies process. 
Then $\eta^j=\Pi(\heta^j)$.
\end{lem}

\begin{proof}[Sketch of proof of \cref{lem:jumpj}] 
Note that this property does not involve the values of the
parameters $\var_j$ at all, so that the same argument as in the 
homogeneous case applies. There are various approaches,
each of which results in a certain amount of case-checking. 
We outline briefly the argument used 
in \cite{ferrari-martin-2006} which efficiently
reduces the cases to be checked.

The approach is by induction on the number of lines. 
If we already have the result for an $(s-1)$-line diagram,
we can consider the effect of a bell in the $s$-line multiline process 
as a combination of a single species PushTASEP jump on the bottom row,
and a multispecies PushTASEP jump on the labelled configuration in row $2$.  

The other key point is the interpretation of the multispecies
process as a coupling of single species processes, 
as explained at the end of \cref{sec:single}. 
If the effect of the bell ringing at $j$ is correct in all the projections to one-species processess, then it's correct in the multispecies process, by definition. 

Combining these two ingredients, we end up just with the base case
of $s=2$ and a two-line diagram. Then there are just a few cases to check, depending on whether the bell in the bottom row rings at the site of a vacancy or a species-$1$ or a species-$2$ particle, 
and whether the bell on the top row rings at the site of a particle
or a vacancy. The cases are checked in \cite{ferrari-martin-2006} for the case of the process on $\ZZ$, but the details are essentially identical for the process on the ring. 
\end{proof}

From \cref{lem:projection} we immediately obtain the following result, which says that the stationary distribution 
of a state $\eta$ of the multispecies PushTASEP is proportional 
to the sum of weights of the multiline diagrams
with bottom row $\eta$. 

\begin{thm}
\label{thm:stat prob sum}
The stationary probability of $\eta \in \Omega_\lambda$ in 
the multispecies PushTASEP is
given by
\[
\pi(\eta) = 
\sum_{
\substack{\heta\in \hOmega_\lambda:\\
\Pi(\heta)=\eta}
}
\hat{\pi}(\hat{\eta})
=
\frac{1}{Z_\lambda}
\sum_{
\substack{\heta\in \hOmega_\lambda:\\
\Pi(\heta)=\eta}
}
\wt(\heta).
\]
\end{thm}

\begin{proof}[Proof of \cref{thm:multilrep-statprob}]
The result of \cref{thm:multilrep-statprob}
now follows immediately from Proposition 4.1 of 
\cite{corteel-mandelshtam-williams-2022}, which
(in a more general context $t\geq 0$ and $q\geq 0$)
expresses the ASEP polynomial 
$\ASEP_\eta(\var_1, \dots, \var_n; q, t)$ 
as a sum of weights of multiline diagrams
with bottom row given by $\eta$. The form 
of the partition function $Z_\lambda$ is given 
by \eqref{Zlambda}. 
\end{proof}

\section{Observables for the multispecies PushTASEP}
\label{sec:obs}

We generalize formulas for observables in the single species PushTASEP given in \cref{sec:single} to the multispecies case here.
As usual we will fix the content to be defined by $\lambda = \langle 0^{m_0}, \dots, s^{m_s} \rangle$ on $n$ sites.

\subsection{Density and current}
We begin with formulas for the density and current in \cref{prop:curr-singlePushTASEP}.
Recall the Schur polynomials $s_\lambda$ from \eqref{def-schur}. One important formula that we will use repeatedly is the the dual Jacobi-Trudi identity~\cite[Corollary 7.16.2]{stanley-ec2} for the Schur polynomials for partitions with two columns. For these partitions, it is written as
\begin{equation}
\label{dual jt}
s_{\langle 1^b, 2^a \rangle}(x_1, \dots, x_n)
= \det \begin{pmatrix}
e_{a+b}(x_1, \dots, x_n) & e_{a+b+1}(x_1, \dots, x_n) \\
e_{a-1}(x_1, \dots, x_n) & e_a(x_1, \dots, x_n)
\end{pmatrix}
\end{equation}

Recall that $\eta^{(i)}_k$ denotes the occupation variable for the particle of species $i$ at site $k$. By invariance of the model under cyclic permutation of the site lables and rates, it is enough to focus on the density at the first site.
Recall also the notation 
$M_r = m_r + \cdots + m_s$ for $1 \leq r \leq s$.

\begin{prop}
\label{prop:dens}
The density of the particle of species $i$ in the first site in the multispecies PushTASEP with content $\lambda$ is given by
\[
\aver{\eta^{(i)}_1} = \var_1 \frac{s_{\langle 1^{M_{j+1}} , 2^{m_j-1} \rangle}(\var_2,\dots,\var_n)}
{e_{(M_j, M_{j+1})}(\var_1, \dots, \var_n) }.
\]
\end{prop}

\begin{proof}
We will use \cref{prop:colouring}. The density of the particle of species $j$ is the density of the particle of species $1$ in the single species PushTASEP with $M_j$ particles minus the density of the particle of species $1$ in the single species PushTASEP with $M_{j+1}$ particles. We now use \cref{cor:density-singlePushTASEP} and the dual Jacobi-Trudi identity~\eqref{dual jt} to obtain the formula.
\end{proof}

\begin{prop}
\label{prop:curr}
For the multispecies PushTASEP with content $\lambda$, the current of species $j$ for $1 \leq j \leq s$, is given by
\[
\frac{s_{\langle 1^{M_{j+1}}, 2^{m_j-1} \rangle}(\var_1,\dots,\var_n)}{
e_{(M_j, M_{j+1})}(\var_1, \dots, \var_n)}.
\]
\end{prop}

\begin{proof}
The proof follows again from the same coloring used in the proof of \cref{prop:dens}. Using \cref{prop:curr-singlePushTASEP},
the current of particles of species $j$ is
the total current of species $j$ through $s$
 given by $e_{M_j-1}/e_{M_j}$ minus 
 the total current of species $j+1$ through $s$
 given by $e_{M_{j+1}-1}/e_{M_{j+1}}$.
The difference gives the desired formula again using the dual Jacobi-Trudi identity~\eqref{dual jt}.
\end{proof}

\subsection{Two-point correlations}
We now prove the result for the nearest neighbour correlations in \cref{thm:2pt} generalising earlier work of Ayyer--Linusson for the multispecies TASEP~\cite{ayyer-linusson-2016} and related to those of Amir--Angel--Valko for the TASEP speed process~\cite{amir-angel-valko-2011}.

Two-point correlations for the stationary distribution of the  
multispecies
TASEP, which coincides with that for the multispecies
PushTASEP in the case $\var_1=\dots=\var_n$,
have been computed in \cite[Theorem 4.2]{ayyer-linusson-2016}. It is a somewhat tedious exercise to recover the results therein by setting $\var_1 = \cdots = \var_n = 1$ and using the hook-content formula~\cite[Corollary 7.21.4]{stanley-ec2} for two-column partitions. The verification of this is left to the interested reader.

We will follow the strategy of proof in \cite[Theorem 4.2]{ayyer-linusson-2016}. Since the ideas are so similar, we will omit many of the details.
The main idea is to work out the same correlation for the multispecies PushTASEP with three species of particles and use the coloring argument in \cref{prop:colouring}. 
Recall that we are working with the situation where there are $n-1$ species and one particle of each type (and one vacancy) on $n$ sites. 

\begin{prop}
\label{prop:2pt}
Fix $s, t > 0$ such that $s+t < n$.
For the multispecies PushTASEP with content 
$\langle 2^{n-s-t}, \allowbreak
1^t, 0^{s} \rangle$, the joint probability of seeing a vacancy at the first site and a $1$ at the second site is given by
\[
T_1(s, t) = \frac{\var_2 \, s_{\langle 1^{t-1}, 2^s \rangle}(\var_3, \dots, \var_n)}
{e_{(s+t, s)}(\var_1, \dots, \var_n)},
\]
and that of seeing the particle of type $1$ at the first site and a vacancy at the second site is given by
\[
T_2(s, t) = \frac{\var_2 \, s_{\langle 1^{t-1}, 2^s \rangle}(\var_3, \dots, \var_n)
+ \var_1 \var_2 \, s_{\langle 1^{t}, 2^{s-1} \rangle}(\var_3, \dots, \var_n)}
{e_{(s+t, s)}(\var_1, \dots, \var_n)}.
\]
\end{prop}

\begin{proof}
For the three species model, stationary probabilities are computed using multiline diagrams with two rows such that there are a total of $s$ $\occ$'s in the first row and $s + t$ $\occ$'s in the second row. For the first case, we need to sum over the weights of all multiline diagrams of the form
\[
\begin{array}{ccc}
\vac & \vac & \cdots \\
\vac & \occ & \cdots \\
\hline
0 & 1 & \cdots
\end{array}.
\]
Following the strategy in the proof of \cite[Theorem 4.2]{ayyer-linusson-2016}, we see that such configurations are in bijection with semistandard tableaux of shape $\langle 1^{t-1}, 2^{s} \rangle$ and the weights are closely related to the content of the tableaux so that the sum is precisely the Schur polynomial indexed by this partition. For the second case, we sum over two kinds of multiline diagrams, 
\[
\begin{array}{ccc}
\vac & \vac & \cdots \\
\occ & \vac & \cdots \\
\hline
1 & 0 & \cdots
\end{array}
\quad
\text{and}
\quad
\begin{array}{ccc}
\vac & \occ & \cdots \\
\occ & \vac & \cdots \\
\hline
1 & 0 & \cdots
\end{array},
\]
with the same restriction. By similar analysis as above, the sums of weights of such configurations turn out to be Schur polynomials $s_{\langle 1^{t-1}, 2^s \rangle}(\var_3, \dots, \var_n)$ and $s_{\langle 1^{t}, 2^{s-1} \rangle}(\var_3, \dots, \var_n)$ respectively.
\end{proof}

\begin{proof}[Proof of \cref{thm:2pt}]
Clearly $\aver{\eta^{(i)}_1 \eta^{(i)}_2} = 0$ since there is only one particle of each type.
First, we consider $j < i$. Using \cref{prop:2pt}, the coloring argument in \cref{prop:colouring} and inclusion-exclusion principle, we find that 
\begin{multline*}
\aver{\eta^{(j)}_1 \eta^{(i)}_2} = T_1(n-i-1, i-j) 
- T_1(n-i-1, i-j+1) \\
- T_1(n-i, i-j-1) + T_1(n-i,i-j),
\end{multline*}
where $T_1$ is given in \cref{prop:2pt}. The denominator is a symmetric polynomial in $\var_1, \dots, \var_n$, while the numerator is a symmetric polynomial only in $\var_3, \dots, \var_n$. Taking a common denominator and simplifying leads, after using the dual Jacobi-Trudi identity~\eqref{dual jt} and a lot of tedious manipulation, to the desired result.

For $j > i$, we have a similar inclusion-exclusion formula,
\begin{multline*}
\aver{\eta^{(j)}_1 \eta^{(i)}_2} = T_2(n-j-1, j-i) 
- T_2(n-j-1, j-i+1) \\
- T_2(n-j, j-i-1) + T_2(n-j,j-i).
\end{multline*}
Here, again after many calculations in the same vein, we get the stated answer involving $g(j,i)$. The only difference is that when $j = i+1$, there is a compensatory factor, which is also shown. The details are left to the interested reader. 
\end{proof}

\section*{Acknowledgements}
We thank Gidi Amir, Pablo Ferrari, Omer Angel, Leonid Petrov and Lauren Williams for very helpful discussions. 
The first author (AA) acknowledges support from SERB Core grant CRG/2021/001592 and the DST FIST program - 2021 [TPN - 700661].

\bibliography{lrep}
\bibliographystyle{alpha}

\end{document}